\documentclass[preprint,12pt]{elsarticle}

\usepackage[T1]{fontenc}
\usepackage[utf8]{inputenc}
\usepackage{amsmath}
\usepackage{amssymb}
\usepackage{amsthm}
\usepackage{mathtools}
\usepackage{booktabs}
\usepackage{array}
\usepackage{tikz}
\usetikzlibrary{arrows.meta,positioning}
\usepackage{placeins}
\usepackage[hidelinks]{hyperref}

\newcommand{\one}{\mathbf{1}}
\newcommand{\so}{\mathfrak{so}}
\newcommand{\res}{\mathfrak{a}}
\newcommand{\Cons}{\mathcal{C}}

\newcommand{\Umb}{\operatorname{Umb}}
\newcommand{\ExpH}{\operatorname{Exp}_{\odot}}
\newcommand{\LogH}{\operatorname{Log}_{\odot}}
\newcommand{\normF}[1]{\left\lVert #1\right\rVert_F}
\newcommand{\ip}[2]{\left\langle #1,#2\right\rangle_F}
\newcommand{\R}{\mathbb{R}}

\theoremstyle{plain}
\newtheorem{theorem}{Theorem}[section]
\newtheorem{proposition}[theorem]{Proposition}
\newtheorem{corollary}[theorem]{Corollary}
\theoremstyle{definition}
\newtheorem{algorithm}{Algorithm}
\theoremstyle{remark}

\journal{Applied Mathematics and Computation}
\biboptions{sort&compress}

\begin{document}

\begin{frontmatter}

\title{Recursive Multilevel Decomposition of Pairwise Comparison Matrices: Algorithms and Inconsistency Diagnostics}

\author[yiu]{Mert {\c{C}}arbo{\u{g}}a\corref{cor1}}
\ead{mert.carboga@yuksekihtisas.edu.tr}

\author[laurentian]{Waldemar W. Koczkodaj}
\ead{wkoczkodaj@laurentian.ca}

\author[ankara]{Yusuf Yayl{\i}}
\ead{yayli@science.ankara.edu.tr}

\cortext[cor1]{Corresponding author}

\address[yiu]{Y{\"u}ksek \.{I}htisas University, Ankara, T{\"u}rkiye}
\address[laurentian]{Department of Computer Science,
Laurentian University, Sudbury, Ontario, Canada}
\address[ankara]{Department of Mathematics, Faculty of Science,
Ankara University, Ankara, T{\"u}rkiye}

\begin{abstract}
Positive reciprocal pairwise comparison matrices are studied in entrywise logarithmic coordinates. Starting from the established consistent--residual split, we show that, for any fixed ordering, the residual matrix admits a unique top-down decomposition into mutually Frobenius-orthogonal layers. The construction applies to every positive reciprocal PC matrix of size $n\geq3$, without additional spectral, sparsity, or genericity assumptions. An explicit active-matrix recursion and an equivalent cumulative parameter-extraction procedure have computational complexities $O(n^3)$ and $O(n^2)$, respectively. The squared Frobenius norms of the layers define an ordering-conditioned residual-energy profile that provides structural information complementary to the geometric-mean priorities and the Koczkodaj inconsistency indicator (Kii). Entrywise exponentiation yields an exact Hadamard-product reconstruction of the original reciprocal matrix. Numerical experiments confirm agreement between the two computational procedures, exact reconstruction, mutual orthogonality, and numerical accuracy to machine precision.
\end{abstract}

\begin{keyword}
pairwise comparison matrices
\sep multilevel decomposition
\sep inconsistency analysis
\sep Koczkodaj inconsistency indicator
\sep geometric mean
\sep Frobenius orthogonality
\MSC[2020] 15A23 \sep 22E60 \sep 15A60 \sep 91B06
\end{keyword}

\end{frontmatter}

\section{Introduction}
\label{sec:introduction}

Pairwise comparisons encode relative information among a set of
entities $E_1,\ldots,E_n$. In the multiplicative setting, a positive
entry $m_{ij}$ represents the assessed ratio of $E_i$ to $E_j$, and the
collection of these assessments forms a pairwise comparison matrix
$M=(m_{ij})$ \cite{koczkodaj2018axiomatization}. The present work considers the reciprocal case,
\[
m_{ii}=1,
\qquad
m_{ji}=m_{ij}^{-1}.
\]
Thus, the matrix records relative rather than absolute information. In
particular, if the assessments arise from positive scores
$s_1,\ldots,s_n$ through
\[
m_{ij}=\frac{s_i}{s_j},
\]
then multiplying all scores by the same positive constant leaves every
pairwise ratio unchanged.

For $n$ entities, reciprocity leaves $n(n-1)/2$ distinct pairwise
assessments. If these ratios are generated by a single underlying score
vector, they satisfy the multiplicative consistency condition
\[
m_{ij}m_{jk}=m_{ik}
\qquad
\text{for all } i,j,k.
\]
When pairwise assessments are supplied separately, this condition need
not hold, and inconsistency arises naturally. The resulting problem is
therefore not only to extract a consistent ranking component but also
to characterize the structure of the departure from consistency.

The present paper focuses on multiplicative reciprocal pairwise
comparisons; additive coordinates enter through the entrywise logarithm.
Applying
\[
A=\LogH(M)=(\log m_{ij})
\]
gives $A\in\mathfrak{so}(n)$, while multiplicative consistency becomes
the additive condition
\[
a_{ij}+a_{jk}=a_{ik},
\]
equivalently $a_{ij}=x_i-x_j$ for a score vector $x$. Logarithmic
coordinates therefore separate the consistency problem from the
multiplicative representation. The logarithmic transformation is
bijective on positive reciprocal matrices, so no information is lost at
the level of representation \cite{KoczkodajMarekYayli2021}. This should not be interpreted as
asserting that an optimum defined in the additive space remains optimal
after nonlinear back-transformation.

Established methods supply complementary information about pairwise data.
Row geometric means yield a priority vector \cite{CrawfordWilliams1985}, while
the Koczkodaj inconsistency indicator (Kii) summarizes departure from
consistency through triads
\cite{koczkodaj1993new,koczkodaj2018axiomatization}. Kii provides a triad-based
alternative to eigenvalue-based consistency measures; the mathematical
distinctions are discussed in \cite{KoczkodajSzwarc2014}. Orthogonal
projection methods compute a global consistent approximation and its residual component
\cite{KoczkodajSmarzewskiSzybowski2020,SmarzewskiRutka2020,
KoczkodajMarekYayli2021}, while graph- and Hodge-theoretic formulations interpret
pairwise data through edge flows and cyclic components
\cite{JiangLimYaoYe2011}. Computational orthogonalization provides
efficient constructions of the global split
\cite{BenitezKoczkodajKowalczyk2024}. These approaches identify a ranking, a
global inconsistency measure, or a global residual matrix, but do not by themselves
decompose the residual component into a hierarchy of mutually orthogonal
layers along a fixed nested sequence of dimensions.

The established Frobenius-orthogonal decomposition is
\begin{equation}
\so(n)=\Cons_n\overset{\perp}{\oplus}\res_n,
\qquad
\res_n=\{U\in\so(n):U\one_n=0\}.
\label{eq:intro-global-split}
\end{equation}
Here $\Cons_n$ denotes the additive consistent subspace, while $\res_n$ is
the zero-row-sum skew-symmetric residual space.
Accordingly, $A=C+U$ with $C\in\Cons_n$ and $U\in\res_n$.
We call $U=A-C$ the global residual matrix. Equivalently, $U$ is the
Frobenius-orthogonal component of $A$ relative to the consistent subspace
$\Cons_n$. This established global consistent--residual split serves as the starting point of our construction and is not
claimed here as a new result.
The problem considered here is whether $U$ itself can be decomposed
recursively while preserving exact reconstruction, uniqueness for a fixed
ordering, and Frobenius orthogonality.

\begingroup
\emergencystretch=1em
The residual space $\res_n$ also has a geometric interpretation: it is the Lie
algebra of the umbrella group, the special-orthogonal stabilizer of the
all-ones direction. Umbrella matrices originate in classical studies of
orthogonal transformations and were subsequently developed from Lie-group and
differential-geometric viewpoints
\cite{Alisbah1976,Hacisalihoglu1977,Ozdamar1977,Kuruoglu1983,Esin1986}.
More recent work connected this framework with pairwise comparison matrices
\cite{CarbogaYayli2023PC} and extended it to Lorentzian, Galilean, and
circulant settings \cite{CarbogaYayli2025Lorentz,CarbogaYayli2025Galilean,
CarbogaYayli2026Circulant}. This geometric lineage explains the terminology,
whereas the recursive decomposition below is matrix-theoretic and requires no
Lie-group machinery in its derivation.
\par
\endgroup

For a fixed ordering of the alternatives, the active indices determine
the top-down chain $n\to n-1\to\cdots\to3$. Every $U\in\res_n$ then admits a
unique decomposition into mutually Frobenius-orthogonal layers
\begin{equation}
U=L_n+L_{n-1}+\cdots+L_3.
\label{eq:intro-multilevel}
\end{equation}
The layer $L_k$ is the component extracted when the active dimension decreases
from $k$ to $k-1$. Its parameter vector has $k-2$ independent coordinates,
and
\[
\sum_{k=3}^{n}(k-2)=\frac{(n-1)(n-2)}{2}=\dim\res_n,
\]
so the total parameter count matches the dimension of the global residual
space.

No additional spectral, sparsity, or genericity condition is required: once
an ordering is fixed, the construction applies to every positive reciprocal
pairwise comparison matrix of size $n\geq3$. In the consistent case, the
global residual matrix $U$ is zero and all multilevel layers vanish.

Two equivalent computational realizations make the construction explicit.
The explicit active-matrix realization described in
Algorithm~\ref{alg:explicit} materializes and updates a
$(k-1)\times(k-1)$ active residual matrix at level $k$ and has computational
complexity $O(n^3)$. The compact cumulative procedure instead reads one active
column per level and extracts
all layer parameters in $O(n^2)$ operations, including the global split. The
latter bound concerns parameter extraction and does not include explicitly
materializing every full $n\times n$ layer matrix; its quadratic complexity is
of the same asymptotic order as the size of the dense $n\times n$ input.

Mutual orthogonality yields the energy identity
\[
\normF{U}^2=\sum_{k=3}^{n}\normF{L_k}^2.
\]
For $U\ne0$, the shares $p_k=\normF{L_k}^2/\normF{U}^2$ define a
\emph{nested-chain residual-energy profile}. Simultaneous relabeling preserves
the total residual energy but generally changes the individual shares.
Consequently, the residual-energy profile is an ordering-conditioned structural diagnostic,
not a permutation-invariant inconsistency index or a causal attribution of
inconsistency to individual alternatives.

Logarithmic triad-defect coordinates provide a complementary interpretation:
the Frobenius decomposition resolves an $\ell^2$-type residual energy, whereas
Kii records the $\ell^\infty$ magnitude of the worst triad defect. A worst
triad cannot in general be assigned uniquely to a single layer. The consistent
component $C$ recovers the geometric-mean priority vector, so recursively
resolving $U$ does not alter the associated ranking.

Entrywise exponentiation transfers the additive expansion back to the
reciprocal matrix and gives an exact Hadamard-product reconstruction. The
numerical experiments verify exact reconstruction and mutual orthogonality, confirm
the equivalence of the explicit and cumulative procedures, compare their
computational behavior, and illustrate the diagnostic interpretation and its
ordering dependence. Section~\ref{sec:global-residual} reviews logarithmic
coordinates and the global split; Section~\ref{sec:recursive} develops the
recursion and algorithms; Sections~\ref{sec:diagnostics} and
\ref{sec:reconstruction} treat diagnostics and multiplicative reconstruction;
Section~\ref{sec:numerics} presents the computations; and
Section~\ref{sec:conclusions} concludes the paper.
\section{Pairwise comparison matrices and the global residual matrix}
\label{sec:global-residual}

\subsection{Reciprocal PC matrices and logarithmic coordinates}
\label{subsec:log-coordinates}

Let
\[
M=(m_{ij})\in\R_{>0}^{n\times n},
\qquad
m_{ii}=1,
\qquad
m_{ij}m_{ji}=1.
\]
Such an $M$ is a positive reciprocal pairwise comparison matrix.
Multiplicative consistency means
\[
m_{ij}m_{jk}=m_{ik}
\qquad\text{for all }i,j,k.
\]
Define the entrywise logarithm and exponential by
\[
A=\LogH(M)=(\log m_{ij}),
\qquad
M=\ExpH(A)=(e^{a_{ij}}).
\]
Both operations act entrywise and are distinct from the ordinary matrix
logarithm and exponential.
Reciprocity gives $A^T=-A$, so $A\in\so(n)$. Conversely, every
$A\in\so(n)$ produces a positive reciprocal pairwise comparison matrix
through $\ExpH$. Multiplicative consistency becomes the additive condition
\[
a_{ij}+a_{jk}=a_{ik}.
\]
The Hadamard product is $(M\odot N)_{ij}=m_{ij}n_{ij}$. Ordinary addition in
logarithmic coordinates corresponds to Hadamard multiplication because
entrywise exponentiation satisfies
\begin{equation}
\ExpH(A+B)=\ExpH(A)\odot\ExpH(B).
\label{eq:entrywise-addition}
\end{equation}

The matrix space is equipped with the Frobenius inner product and norm,
\[
\ip{A}{B}=\operatorname{tr}(A^TB),
\qquad
\normF{A}=\sqrt{\ip{A}{A}}.
\]

\subsection{Consistency and the global orthogonal decomposition}
\label{subsec:global-orthogonal}

Let $\one_n=(1,\ldots,1)^T$, $J_n=\one_n\one_n^T$, and
\[
H_n=I_n-\frac1nJ_n.
\]
The additive consistent subspace is
\[
\Cons_n=
\{C_n(x)=x\one_n^T-\one_nx^T:x\in\one_n^\perp\},
\]
and the global residual space is
\[
\res_n=\{U\in\so(n):U\one_n=0\}.
\]
The following established decomposition is the starting point of the
recursive construction
\cite{KoczkodajSmarzewskiSzybowski2020,SmarzewskiRutka2020,KoczkodajMarekYayli2021}.

\begin{theorem}
\label{thm:global}
For every $A\in\so(n)$, define
\[
x=\frac1nA\one_n,
\qquad
C=C_n(x),
\qquad
U=A-C.
\]
Then
\[
C\in\Cons_n,
\qquad
U\in\res_n,
\qquad
\ip{C}{U}=0.
\]
Hence
\[
\so(n)=\Cons_n\overset{\perp}{\oplus}\res_n.
\]
The projections also satisfy
\begin{equation}
C=\frac1n(AJ_n+J_nA),
\qquad
U=H_nAH_n.
\label{eq:global-projections}
\end{equation}
\end{theorem}

\begin{proof}
Skew-symmetry first gives
\[
\one_n^TA\one_n=-(\one_n^TA\one_n)^T=0,
\]
and therefore
\[
\one_n^Tx=\frac1n\one_n^TA\one_n=0.
\]
Thus $x\in\one_n^\perp$ and $C=C_n(x)\in\Cons_n$. Next,
\[
C\one_n=nx-\one_n(x^T\one_n)=A\one_n,
\]
so $U\one_n=0$ and hence $U\in\res_n$. For any
$C_n(y)\in\Cons_n$ and $U\in\res_n$,
\[
\ip{C_n(y)}{U}
=\operatorname{tr}\!\bigl((\one_ny^T-y\one_n^T)U\bigr)=0,
\]
because $U\one_n=0$ and $\one_n^TU=0$. Uniqueness follows from
orthogonality of the two subspaces. Finally, substituting
$x=n^{-1}A\one_n$ gives
\[
C=\frac1n(AJ_n+J_nA).
\]
Since $\one_n^TA\one_n=0$, expansion of $H_nAH_n$ then yields
$U=A-C=H_nAH_n$, which proves \eqref{eq:global-projections}.
\end{proof}

The orthogonality gives
\begin{equation}
\normF{A}^2=\normF{C}^2+\normF{U}^2.
\label{eq:global-energy}
\end{equation}
This identity quantifies the total residual energy but does not reveal its internal
distribution.

\subsection{The residual space and its umbrella interpretation}
\label{subsec:umbrella-residual}

The umbrella group is
\[
\Umb(n)=\{Q\in SO(n):Q\one_n=\one_n\}.
\]
It fixes the all-ones direction, and its Lie algebra is the matrix space
$\res_n$. This is the only group-theoretic fact needed later; the recursive
derivation itself is matrix-theoretic. The residual-space dimension is
\[
\dim\res_n=\frac{(n-1)(n-2)}2.
\]
The recursive construction follows the active-size chain
\[
\res_n\longrightarrow\res_{n-1}\longrightarrow\cdots
\longrightarrow\res_3.
\]
Throughout the construction, this chain is understood relative to the
fixed ordering of the alternatives; a different ordering produces a different
sequence of leading principal active blocks.
\section{Recursive multilevel decomposition}
\label{sec:recursive}

\subsection{One-step top-down reduction}
\label{subsec:one-step}

At an active size $k\geq3$, write
\[
U^{(k)}=
\begin{pmatrix}
B^{(k)}&a^{(k)}\\
-(a^{(k)})^T&0
\end{pmatrix}
\in\res_k,
\qquad m=k-1.
\]
The zero-row-sum condition gives
\[
B^{(k)}\one_m+a^{(k)}=0.
\]
Set
\[
v_k=\frac1mB^{(k)}\one_m,
\qquad a^{(k)}=-mv_k.
\]
Because $B^{(k)}$ is skew-symmetric,
\[
\one_m^Tv_k
=\frac1m\one_m^TB^{(k)}\one_m=0,
\]
so $v_k\in\one_m^\perp$.
For $v\in\one_m^\perp$, define
\[
G_k(v)=
\begin{pmatrix}
v\one_m^T-\one_mv^T&-mv\\
mv^T&0
\end{pmatrix},
\qquad
\iota_k(W)=
\begin{pmatrix}
W&0\\0&0
\end{pmatrix},
\qquad W\in\res_{k-1}.
\]
The matrix $G_k(v)$ is skew-symmetric, and
\[
G_k(v)\one_k=
\begin{pmatrix}
m v-\one_m(v^T\one_m)-mv\\
m v^T\one_m
\end{pmatrix}
=0.
\]
Hence $G_k(v)\in\res_k$ whenever $v\in\one_m^\perp$.

\begin{theorem}
\label{thm:one-step}
Every $U^{(k)}\in\res_k$ admits a unique Frobenius-orthogonal
decomposition
\begin{equation}
U^{(k)}=G_k(v_k)+\iota_k(U^{(k-1)}),
\label{eq:one-step}
\end{equation}
where
\[
v_k=\frac1{k-1}B^{(k)}\one_{k-1}
\]
and
\begin{equation}
U^{(k-1)}=
B^{(k)}-
\bigl(v_k\one_{k-1}^T-\one_{k-1}v_k^T\bigr)
\in\res_{k-1}.
\label{eq:active-update}
\end{equation}
\end{theorem}

\begin{proof}
By the preceding observation,
\[
v_k\in\one_{k-1}^\perp
\quad\text{and}\quad
G_k(v_k)\in\res_k.
\]
The matrix in
\eqref{eq:active-update} is skew-symmetric, and
\[
U^{(k-1)}\one_{k-1}
=B^{(k)}\one_{k-1}-(k-1)v_k=0.
\]
Therefore $U^{(k-1)}\in\res_{k-1}$. The leading blocks in
\eqref{eq:one-step} sum to $B^{(k)}$, while the last column of $G_k(v_k)$ is
$-(k-1)v_k=a^{(k)}$; hence the formulas reconstruct $U^{(k)}$. For every
$W\in\res_{k-1}$,
\[
\ip{G_k(v_k)}{\iota_k(W)}
=\left\langle v_k\one_{k-1}^T-\one_{k-1}v_k^T,W\right\rangle_F=0.
\]
Indeed, $W\one_{k-1}=0$ and $\one_{k-1}^TW=0$. Finally, $v_k$ is determined
by $(k-1)^{-1}B^{(k)}\one_{k-1}$, equivalently by the active last column
$a^{(k)}=-(k-1)v_k$; once $v_k$ is known, $U^{(k-1)}$ is determined by subtraction.
This proves uniqueness.
\end{proof}

Direct calculation gives the layer-energy formula
\begin{equation}
\normF{G_k(v)}^2=2k(k-1)\lVert v\rVert_2^2.
\label{eq:local-layer-energy}
\end{equation}
Thus one reduction step separates the newly extracted $k$-level component
from the embedded copy of the lower-dimensional active residual matrix.

\subsection{Full multilevel decomposition}
\label{subsec:full-multilevel}

For $3\leq k\leq n$, define
\[
\iota_{k,n}(X)=
\begin{pmatrix}
X&0\\0&0_{n-k}
\end{pmatrix},
\qquad
L_k=\iota_{k,n}(G_k(v_k)).
\]

\begin{theorem}
\label{thm:full-multilevel}
For a fixed ordering of the alternatives, every $U\in\res_n$ admits a
unique decomposition
\begin{equation}
U=L_n+L_{n-1}+\cdots+L_3,
\label{eq:full-residual}
\end{equation}
and the layers are mutually Frobenius orthogonal. Consequently, every
$A\in\so(n)$ has the expansion
\begin{equation}
A=C_n(x)+L_n+L_{n-1}+\cdots+L_3.
\label{eq:full-additive}
\end{equation}
The multilevel part is unique relative to the nested chain induced by the
fixed ordering.
The recursion terminates at $k=3$ because $\res_2=\{0\}$.
\end{theorem}

\begin{proof}
Apply Theorem~\ref{thm:one-step} successively from active size $n$ to active
size $3$ along the fixed chain. At level $k$, the extracted matrix $G_k(v_k)$
is orthogonal to the embedded residual space $\iota_k(\res_{k-1})$. Every
subsequent lower-level layer lies in that embedded space, so the $k$-level
component is orthogonal to every layer extracted later. The leading-block
embedding into dimension $n$ preserves the Frobenius inner product; hence all
$L_k$ are mutually orthogonal. The recursion stops because $\res_2=\{0\}$,
and uniqueness follows successively from Theorem~\ref{thm:one-step} for the
fixed chain. Combining the residual decomposition with
Theorem~\ref{thm:global} gives \eqref{eq:full-additive}.
\end{proof}

The layer parameters account for all residual degrees of freedom:
\[
\sum_{k=3}^n(k-2)=\frac{(n-1)(n-2)}2=\dim\res_n.
\]
Together with the $n-1$ parameters of the consistent component, this gives
\[
(n-1)+\sum_{k=3}^n(k-2)=\frac{n(n-1)}2=\dim\so(n).
\]

Figure~\ref{fig:multilevel-residual-decomposition} summarizes the recursive
mechanism underlying Theorem~\ref{thm:full-multilevel}. The original
pairwise comparison matrix $M$ remains unchanged throughout the construction.
After the global split $A=C_n(x)+U$, only the active residual matrix is
propagated through the nested dimensions: at level $k$, the
Frobenius-orthogonal layer $L_k$ is extracted and the remainder becomes
$U^{(k-1)}$.

\begin{figure}[!htbp]
\centering
\begin{tikzpicture}[
  node distance=5mm and 9mm,
  flowbox/.style={draw, align=center, inner sep=3pt,
                  text width=0.50\linewidth},
  activebox/.style={draw, align=center, inner sep=3pt,
                    text width=0.19\linewidth},
  relationbox/.style={draw, align=left, inner sep=4pt,
                      text width=0.40\linewidth},
  summarybox/.style={draw, align=center, inner sep=4pt,
                     text width=0.58\linewidth},
  flowarrow/.style={-{Latex[length=2.2mm,width=1.6mm]}, line width=0.6pt},
  relationline/.style={dashed, line width=0.5pt}
]
\node[flowbox] (M) {\textbf{Positive reciprocal PC matrix $M$}\\[-2pt]
  \emph{original matrix remains fixed}};
\node[flowbox, below=of M] (A) {$A=\LogH(M)$};
\node[flowbox, below=of A] (split) {$A=C_n(x)+U$};
\node[activebox, below=of split] (Un) {Active residual matrix\\
  $U^{(n)}=U$};
\node[activebox, below=of Un] (Unm) {Active residual matrix\\
  $U^{(n-1)}$};
\node[below=3mm of Unm] (dots) {$\vdots$};
\node[activebox, below=3mm of dots] (Uthree) {Active residual matrix\\
  $U^{(3)}$};
\node[relationbox, right=of Unm] (relation) {\textbf{General level $k$}\\[2pt]
  $\displaystyle U^{(k)}=G_k(v_k)+\iota_k\!\left(U^{(k-1)}\right)$\\[3pt]
  $\displaystyle L_k=\iota_{k,n}\!\left(G_k(v_k)\right)$};
\node[summarybox, below=5mm of Uthree] (summary) {
  $\displaystyle U=L_n+L_{n-1}+\cdots+L_3$\\[2pt]
  \emph{mutually Frobenius orthogonal}\\[-1pt]
  $\displaystyle \ip{L_j}{L_k}=0,\qquad j\ne k$};

\draw[flowarrow] (M) -- node[right, font=\small] {entrywise logarithm} (A);
\draw[flowarrow] (A) -- node[right, font=\small] {global orthogonal split} (split);
\draw[flowarrow] (split) -- (Un);
\draw[flowarrow] (Un) -- node[left, font=\small] {extract $L_n$} (Unm);
\draw[flowarrow] (Unm) -- node[left, font=\small] {extract $L_{n-1}$} (dots);
\draw[flowarrow] (dots) -- (Uthree);
\draw[flowarrow] (Uthree) -- node[right, font=\small] {extract $L_3$} (summary);
\draw[relationline] (Unm.east) -- (relation.west);
\end{tikzpicture}
\caption{Recursive multilevel decomposition of the global residual matrix.}
\label{fig:multilevel-residual-decomposition}
\end{figure}
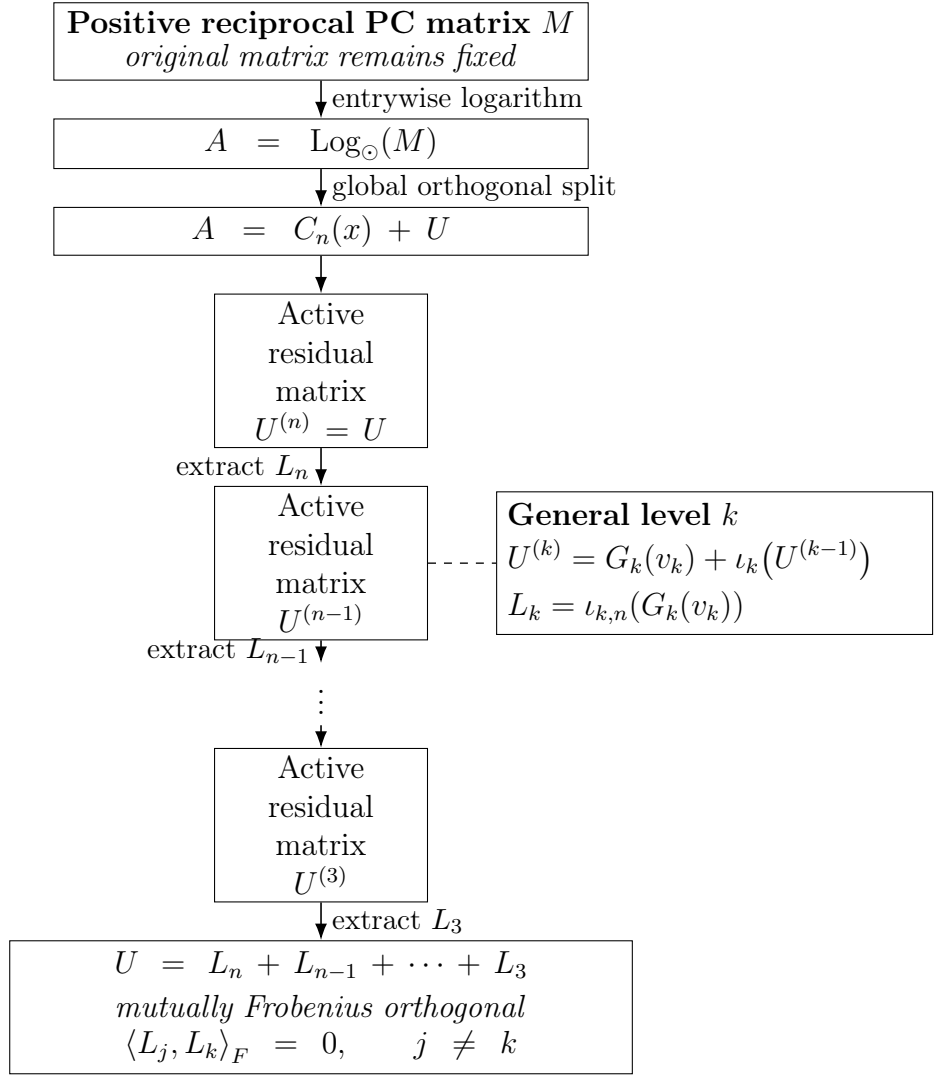
\FloatBarrier

\subsection{Explicit recursive extraction algorithm}
\label{subsec:explicit-algorithm}

The decomposition can be computed without solving a global linear system.

\begin{algorithm}
\label{alg:explicit}
\textbf{Input:} $A\in\so(n)$.

\textbf{Procedure:}
\begin{enumerate}
\item Compute $x\leftarrow n^{-1}A\one_n$,
$C\leftarrow C_n(x)$, and $U^{(n)}\leftarrow A-C$.
\item For $k=n,n-1,\ldots,3$:
  \begin{enumerate}
  \item let $B^{(k)}$ be the leading $(k-1)\times(k-1)$ block of
  $U^{(k)}$;
  \item set $v_k\leftarrow(k-1)^{-1}B^{(k)}\one_{k-1}$;
  \item if $k>3$, set
  \[
  U^{(k-1)}\leftarrow B^{(k)}-
  \bigl(v_k\one_{k-1}^T-\one_{k-1}v_k^T\bigr).
  \]
  \end{enumerate}
\end{enumerate}
\textbf{Output:} $x,v_n,v_{n-1},\ldots,v_3$.
\end{algorithm}

In the explicit implementation of Algorithm~\ref{alg:explicit}, the active
$(k-1)\times(k-1)$ residual matrix is represented and updated at every level. The
work at level $k$ is therefore $O(k^2)$, and
\[
\sum_{k=3}^n O(k^2)=O(n^3).
\]
This bound applies to the active-matrix realization written in
Algorithm~\ref{alg:explicit}; it is not a lower bound for every implementation
of the recurrence.

\subsection{Compact implementation and computational cost}
\label{subsec:compact}

The active last column gives a compact alternative. Since
$U^{(k)}\one_k=0$,
\[
v_k=-\frac1{k-1}a^{(k)}.
\]
Let $U_{[k]}$ denote the leading $k\times k$ block of the original global
residual matrix $U$, initialize $s^{(n)}=0$, and suppose
\begin{equation}
U^{(k)}=U_{[k]}-
\bigl(s^{(k)}\one_k^T-\one_k(s^{(k)})^T\bigr).
\label{eq:cumulative-assumption}
\end{equation}
For $i<k$ this gives
\[
a_i^{(k)}=u_{ik}-\bigl(s_i^{(k)}-s_k^{(k)}\bigr),
\]
so the compact update is
\begin{equation}
(v_k)_i=-\frac{u_{ik}-s_i^{(k)}+s_k^{(k)}}{k-1},
\qquad
s^{(k-1)}=s_{1:k-1}^{(k)}+v_k.
\label{eq:compact-update}
\end{equation}

\begin{proposition}
\label{prop:compact}
The cumulative procedure \eqref{eq:compact-update} produces exactly the same
parameters $v_n,\ldots,v_3$ as Algorithm~\ref{alg:explicit}.
\end{proposition}

\begin{proof}
At $k=n$, \eqref{eq:cumulative-assumption} holds with $s^{(n)}=0$.
Assume that it holds at active size $k$; thus $s^{(k)}$ records the cumulative
correction vector that defines the consistent-form term applied to the leading
block $U_{[k]}$ of the original global residual matrix. This term converts $U_{[k]}$ into the current active
residual matrix $U^{(k)}$. Its last column and the zero-row-sum condition give
$v_k=-(k-1)^{-1}a^{(k)}$, which is the first formula in
\eqref{eq:compact-update} and agrees with Theorem~\ref{thm:one-step}. Applying
the leading-block update yields
\[
U^{(k-1)}
=U_{[k-1]}-
C_{k-1}\!\left(s_{1:k-1}^{(k)}+v_k\right),
\]
where $C_{k-1}(z)=z\one_{k-1}^T-\one_{k-1}z^T$. Thus the second formula in
\eqref{eq:compact-update} establishes the induction step.
\end{proof}

At step $k$, the compact parameter update reads one active column and costs
$O(k)$ operations. Hence compact extraction of all parameters costs
$O(n^2)$, including the initial $O(n^2)$ global split. This statement concerns
parameter extraction; explicitly materializing every full layer matrix is a
separate output operation. No claim about numerical stability is made on the basis of the
operation count. Thus the cumulative formulation reduces parameter-extraction
complexity from $O(n^3)$ for the explicit active-matrix recursion to $O(n^2)$,
while producing exactly the same layer parameters.
\section{Multilevel inconsistency diagnostics}
\label{sec:diagnostics}

\subsection{Frobenius residual-energy profile}
\label{subsec:energy-profile}

The mutual orthogonality in Theorem~\ref{thm:full-multilevel} gives
\begin{equation}
\normF{U}^2=\sum_{k=3}^n\normF{L_k}^2.
\label{eq:residual-energy}
\end{equation}
Define
\[
E_k=\normF{L_k}^2=2k(k-1)\lVert v_k\rVert_2^2.
\]
For $U\ne0$, the normalized nested-chain residual-energy profile is
\begin{equation}
p_k=\frac{E_k}{\normF{U}^2},
\qquad
p_k\geq0,
\qquad
\sum_{k=3}^np_k=1.
\label{eq:normalized-profile}
\end{equation}
The residual-energy profile describes how the fixed global residual energy $\normF{U}^2$ is
distributed across the layers of the fixed nested chain. It is not a
stand-alone inconsistency measure and should be interpreted together with the
global residual magnitude $\normF{U}$; its interpretation is structural
rather than causal.

\subsection{Dependence on the ordering}
\label{subsec:ordering}

For a permutation matrix $P$, let
\[
A^P=PAP^T,
\qquad
x^P=Px,
\qquad
U^P=PUP^T.
\]
The Frobenius norm is invariant under this simultaneous relabeling, so
\begin{equation}
\normF{U^P}=\normF{U}.
\label{eq:permutation-invariance}
\end{equation}
Thus simultaneous relabeling preserves the total residual energy. The
decomposition, however, uses a leading-principal nested chain, so the
individual layer shares generally depend on the chosen ordering. The
residual-energy profile is therefore an ordering-conditioned structural diagnostic. We make no claim that any ordering is optimal, nor do we propose an order-selection criterion.

\subsection{Geometric-mean priorities and the consistent component}
\label{subsec:gm}

\begin{proposition}
\label{prop:gm}
Let $A=\LogH(M)$ and $x=n^{-1}A\one_n$. Then
\[
x_i=\frac1n\sum_{j=1}^n\log m_{ij},
\qquad
e^{x_i}=\left(\prod_{j=1}^nm_{ij}\right)^{1/n}.
\]
Consequently, the normalized geometric-mean priority vector is
\begin{equation}
w_i=\frac{e^{x_i}}{\sum_{\ell=1}^ne^{x_\ell}}.
\label{eq:gm-priorities}
\end{equation}
\end{proposition}

\begin{proof}
The first identity is the $i$-th coordinate of $x=n^{-1}A\one_n$.
Exponentiating the row average of the logarithms gives the second identity
and then \eqref{eq:gm-priorities}.
\end{proof}

Thus $x$ determines the consistent ranking component, while $U$ is the
global residual matrix. Decomposing $U$ into layers does not alter the geometric-mean
ranking.

\subsection{Triad defects and the Koczkodaj inconsistency indicator}
\label{subsec:triads}

For $A\in\so(n)$, define
\[
\bigl(\mathcal T_n(A)\bigr)_{ijk}
=a_{ij}+a_{jk}-a_{ik},
\qquad i<j<k.
\]
The vector $\mathcal T_n(A)$ collects all logarithmic triad defects indexed by
$i<j<k$. Because $\mathcal T_n(C_n(x))=0$, one has
$\mathcal T_n(A)=\mathcal T_n(U)$. Related geometric and norm-based
treatments of cyclic pairwise data are discussed in
\cite{JiangLimYaoYe2011,FedrizziCivolaniCritch2020}.

\begin{theorem}
\label{thm:triad-identity}
For every $U\in\res_n$,
\begin{equation}
\lVert\mathcal T_n(U)\rVert_2^2=\frac n2\normF{U}^2.
\label{eq:triad-identity}
\end{equation}
\end{theorem}

\begin{proof}
For $A\in\so(n)$, direct expansion gives
\[
\sum_{i<j<k}(a_{ij}+a_{jk}-a_{ik})^2
=\frac n2\normF{A}^2-\lVert A\one_n\rVert_2^2.
\]
Since $U\one_n=0$, the assertion follows.
\end{proof}

\begin{corollary}
\label{cor:triad-layers}
For the decomposition $U=L_n+\cdots+L_3$, the vectors
$\mathcal T_n(L_k)$ are mutually orthogonal and
\[
\lVert\mathcal T_n(L_k)\rVert_2^2=\frac n2\normF{L_k}^2.
\]
Therefore
\[
\lVert\mathcal T_n(U)\rVert_2^2
=\sum_{k=3}^n\lVert\mathcal T_n(L_k)\rVert_2^2,
\qquad
p_k=\frac{\lVert\mathcal T_n(L_k)\rVert_2^2}
{\lVert\mathcal T_n(U)\rVert_2^2}.
\]
\end{corollary}

\begin{proof}
Apply Theorem~\ref{thm:triad-identity} by polarization and use the
Frobenius orthogonality of the layers.
\end{proof}

Thus the same normalized layer shares are obtained in Frobenius coordinates
and in triad-defect coordinates.

For a positive reciprocal pairwise comparison matrix, set
\[
\delta_{ijk}=\log\!\left(\frac{m_{ij}m_{jk}}{m_{ik}}\right).
\]
Using the standard Koczkodaj triad formulation
\cite{koczkodaj1993new,koczkodaj2018axiomatization}, its logarithmic form is
\begin{equation}
\operatorname{Kii}(M)
=1-\exp\!\left(-\lVert\mathcal T_n(U)\rVert_\infty\right).
\label{eq:koczkodaj}
\end{equation}
The two diagnostics therefore probe complementary norms of the same global
residual matrix: the residual-energy profile resolves its $\ell^2$ energy, whereas Kii records the
severity of its largest triad defect.

\begin{table}[!ht]
\centering
\caption{Complementary outputs from the same logarithmic decomposition.}
\label{tab:diagnostic-comparison}
\small
\begin{tabular}{>{\raggedright\arraybackslash}p{0.23\linewidth}
                >{\centering\arraybackslash}p{0.31\linewidth}
                >{\raggedright\arraybackslash}p{0.34\linewidth}}
\toprule
Output & Mathematical quantity & Information \\
\midrule
GM priorities & $w_i\propto e^{x_i}$ & relative priorities \\
Kii & $1-e^{-\lVert\mathcal T_n(U)\rVert_\infty}$ & worst-triad severity \\
Nested-chain residual-energy profile & $\normF{U}^2$ and $p_k$ & residual-energy distribution \\
\bottomrule
\end{tabular}
\end{table}

Since
\[
\delta_{ijk}(U)=\sum_{r=3}^n\delta_{ijk}(L_r),
\]
a worst triad cannot in general be assigned uniquely to a single layer.

In a common consistency-driven approach, triad inconsistency values are
computed, a maximally inconsistent triad is identified, and the corresponding
assessments are revised to reduce the inconsistency; the procedure is repeated
until a prescribed tolerance threshold is reached
\cite{JanickiKoczkodaj1998,KoczkodajKosiekSzybowskiXu2015}. A threshold of
$1/3$ has been used in the literature and has been reported to be sufficient for most
applications \cite{KoczkodajKosiekSzybowskiXu2015}. The multilevel
decomposition developed here complements this consistency-driven approach by
providing structural information about the global residual matrix.

\section{Exact multiplicative reconstruction}
\label{sec:reconstruction}

Let the additive expansion from Theorem~\ref{thm:full-multilevel} be
\[
A=C_n(x)+L_n+L_{n-1}+\cdots+L_3.
\]
Define
\[
M_{\mathrm{con}}=\ExpH(C_n(x)),
\qquad
M_k=\ExpH(L_k),\quad k=3,\ldots,n.
\]

\begin{theorem}
\label{thm:multiplicative}
For a fixed ordering of the alternatives, every positive reciprocal
pairwise comparison matrix $M$ admits the exact factorization
\begin{equation}
M=M_{\mathrm{con}}\odot M_n\odot M_{n-1}\odot\cdots\odot M_3.
\label{eq:multiplicative}
\end{equation}
The factors induced by the multilevel decomposition are unique relative to
the fixed nested chain.
\end{theorem}

\begin{proof}
Apply \eqref{eq:entrywise-addition} to the unique additive decomposition of
$A=\LogH(M)$ relative to the fixed chain:
\[
\ExpH(A)=\ExpH(C_n(x))\odot
\ExpH(L_n)\odot\cdots\odot\ExpH(L_3).
\]
This gives \eqref{eq:multiplicative}. If two such chain-induced factorizations
existed, applying the injective map $\LogH$ would give two additive
decompositions of $A$ for the fixed chain, contradicting
Theorem~\ref{thm:full-multilevel}.
\end{proof}

The consistent factor has entries
\[
(M_{\mathrm{con}})_{ij}=e^{x_i-x_j},
\]
so it is generated by the same potential that determines the geometric-mean
priorities. If
\[
M_{\mathrm{res}}=M_n\odot M_{n-1}\odot\cdots\odot M_3,
\]
then
\[
M=M_{\mathrm{con}}\odot M_{\mathrm{res}},
\qquad
\LogH(M_{\mathrm{res}})=U.
\]
Thus the multiplicative reconstruction retains the exact layer information
without introducing a second, distinct residual concept.
\section{Numerical experiments}
\label{sec:numerics}

\subsection{Computational verification of the recursive algorithm}
\label{subsec:verification}

The verification used
\[
A=\frac12(R-R^T),
\]
where the entries of $R$ were independent standard normal variates. With the
fixed seed $20260819$ (YYYYMMDD), ten independent matrices were tested for each
$n\in\{10,20,40,80,160\}$. The explicit recursion (E) and the compact
cumulative recursion (C) were compared using
\[
\varepsilon_{\rm par}
=\max_k\frac{\lVert v_k^{(E)}-v_k^{(C)}\rVert_2}
{1+\lVert v_k^{(E)}\rVert_2},
\]
\[
\varepsilon_{\rm rec}
=\frac{\left\lVert A-C_n(x)-\sum_{k=3}^nL_k\right\rVert_F}
{1+\normF{A}},
\qquad
\varepsilon_{\rm orth}
=\max_{j\ne k}\frac{|\ip{L_j}{L_k}|}
{1+\normF{L_j}\normF{L_k}}.
\]
Table~\ref{tab:benchmark} reports the maximum error over the ten tests and
the median runtime for decomposition only.

\begin{table}[htbp]
\centering
\caption{Maximum scaled errors and median decomposition runtimes over ten
independent matrices. Runtimes are in milliseconds.}
\label{tab:benchmark}
\scriptsize
\begin{tabular}{rccccc}
\toprule
$n$ & $\max\varepsilon_{\rm par}$ & $\max\varepsilon_{\rm rec}$ &
$\max\varepsilon_{\rm orth}$ & median E & median C \\
\midrule
10  & $1.09\times10^{-16}$ & $9.83\times10^{-17}$ & $9.97\times10^{-17}$ & 0.102 & 0.0578 \\
20  & $1.83\times10^{-16}$ & $9.97\times10^{-17}$ & $1.88\times10^{-16}$ & 0.223 & 0.1128 \\
40  & $3.07\times10^{-16}$ & $1.03\times10^{-16}$ & $3.66\times10^{-16}$ & 0.523 & 0.2296 \\
80  & $8.71\times10^{-16}$ & $1.04\times10^{-16}$ & $1.15\times10^{-15}$ & 1.425 & 0.4610 \\
160 & $1.89\times10^{-15}$ & $1.21\times10^{-16}$ & $2.13\times10^{-15}$ & 5.120 & 1.1392 \\
\bottomrule
\end{tabular}
\end{table}

The computations were performed in Python 3.12.13 with NumPy 2.3.5 under
Windows 11 (build 22000) on an Intel64 Family 6 Model 94 processor. These timings are reported only for this particular run and are not used to support an asymptotic
complexity claim.

The maximum scaled error in the table is $2.13\times10^{-15}$. A separate
identity suite on dimensions $3,4,5,8,12,20,40$ also checked the global split,
energy identity, triad identity, permutation invariance, and multiplicative
reconstruction; its maximum scaled error was $9.99\times10^{-16}$. 

\subsection{A synthetic five-alternative comparison}
\label{subsec:five-alternative}

Consider five generic alternatives $T_1,\ldots,T_5$ with logarithmic
comparison matrix
\[
A=\frac1{10}
\begin{pmatrix}
0&-1&4&6&-3\\
1&0&-5&-4&3\\
-4&5&0&7&2\\
-6&4&-7&0&-7\\
3&-3&-2&7&0
\end{pmatrix}.
\]
Its row sums and consistent potential are
\[
A\one_5=
\begin{pmatrix}3/5\\-1/2\\1\\-8/5\\1/2\end{pmatrix},
\qquad
x=
\begin{pmatrix}3/25\\-1/10\\1/5\\-8/25\\1/10\end{pmatrix}.
\]
The global components are
\[
C_5(x)=\frac1{50}
\begin{pmatrix}
0&11&-4&22&1\\
-11&0&-15&11&-10\\
4&15&0&26&5\\
-22&-11&-26&0&-21\\
-1&10&-5&21&0
\end{pmatrix}
\]
and
\[
U^{(5)}=\frac1{50}
\begin{pmatrix}
0&-16&24&8&-16\\
16&0&-10&-31&25\\
-24&10&0&9&5\\
-8&31&-9&0&-14\\
16&-25&-5&14&0
\end{pmatrix}.
\]
The top-down chain $5\to4\to3$ yields
\[
v_5=
\begin{pmatrix}2/25\\-1/8\\-1/40\\7/100\end{pmatrix},
\quad
v_4=
\begin{pmatrix}-1/20\\17/120\\-11/120\end{pmatrix},
\quad
v_3=
\begin{pmatrix}-1/6\\1/6\end{pmatrix},
\]
with active residual matrices
\[
U^{(4)}=\frac1{40}
\begin{pmatrix}
0&-21&15&6\\21&0&-4&-17\\-15&4&0&11\\-6&17&-11&0
\end{pmatrix},
\]
\[
U^{(3)}=\frac13
\begin{pmatrix}0&-1&1\\1&0&-1\\-1&1&0\end{pmatrix}.
\]

\begin{table}[t]
\centering
\caption{Exact layer parameters and residual-energy profile for the
five-alternative example.}
\label{tab:five-alternative-profile}
\small
\begin{tabular}{cclcc}
\toprule
$k$ & $v_k^T$ & $E_k$ & $p_k$ & percentage \\
\midrule
5 & $(2/25,-1/8,-1/40,7/100)$ & $551/500$ & $551/1256$ & $43.87\%$ \\
4 & $(-1/20,17/120,-11/120)$ & $223/300$ & $1115/3768$ & $29.59\%$ \\
3 & $(-1/6,1/6)$ & $2/3$ & $125/471$ & $26.54\%$ \\
\bottomrule
\end{tabular}
\end{table}

Indeed,
\[
\normF{U^{(5)}}^2=\frac{314}{125}
=\frac{551}{500}+\frac{223}{300}+\frac23.
\]
The reciprocal matrix $M=\ExpH(A)$ is
\[
M\approx
\begin{pmatrix}
1&0.905&1.492&1.822&0.741\\
1.105&1&0.607&0.670&1.350\\
0.670&1.649&1&2.014&1.221\\
0.549&1.492&0.497&1&0.497\\
1.350&0.741&0.819&2.014&1
\end{pmatrix}.
\]
The geometric-mean priority vector is
\[
w\approx(0.2217,0.1779,0.2402,0.1428,0.2173)^T,
\]
and hence
\[
T_3>T_1>T_5>T_2>T_4.
\]

The computation gives the following ten triad defects:
\[
\begin{aligned}
\delta_{123}&=-1,&\delta_{124}&=-11/10,&\delta_{125}&=1/2,\\
\delta_{134}&=1/2,&\delta_{135}&=9/10,&\delta_{145}&=1/5,\\
\delta_{234}&=3/5,&\delta_{235}&=-3/5,&\delta_{245}&=-7/5,\\
\delta_{345}&=-1/5.&&
\end{aligned}
\]
Thus
\[
\lVert\mathcal T_5(U)\rVert_\infty=\frac75,
\qquad
\operatorname{Kii}(M)=1-e^{-7/5}\approx0.75340.
\]
The results have distinct interpretations: $w$ supplies priorities; Kii
records the severity of the worst triad, while the largest local defect
identifies the triad attaining that maximum; and the residual-energy profile describes the structural distribution
of residual energy. In particular, the largest
fifth-level contribution does not imply that alternative $T_5$ causes
inconsistency.

Although the present matrix is synthetic, pairwise-comparison methods have
also been applied in clinical data studies, including work on improving
medical-scale predictability \cite{KakiashviliKoczkodajWoodburySmith2012}.

\subsection{Sensitivity of the nested-chain residual-energy profile to the ordering}
\label{subsec:ordering-sensitivity}

All $5!=120$ simultaneous row/column permutations of the same matrix were
enumerated. In every case the total residual energy remained
$314/125=2.512$ up to machine precision, while the nested-chain
residual-energy profiles varied as summarized in
Table~\ref{tab:ordering-summary}.

\begin{table}[t]
\centering
\caption{Residual-energy profile over all 120 simultaneous orderings
(percent).}
\label{tab:ordering-summary}
\small
\begin{tabular}{lrrr}
\toprule
Layer share & Minimum & Median & Maximum \\
\midrule
$100p_5$ & 31.13 & 45.86 & 77.31 \\
$100p_4$ & 1.19  & 33.17 & 67.81 \\
$100p_3$ & 1.06  & 9.55  & 52.02 \\
\bottomrule
\end{tabular}
\end{table}

The range shows that layer percentages depend on the selected ordering. It
does not define an
optimal ordering, and no order-selection rule is proposed.
\section{Conclusions}
\label{sec:conclusions}

The established global consistent--residual decomposition
\cite{KoczkodajSmarzewskiSzybowski2020,SmarzewskiRutka2020,KoczkodajMarekYayli2021}
provides the starting point for this work. For a fixed ordering of the alternatives,
we have developed a recursive multilevel decomposition of the global residual
matrix into unique, mutually Frobenius-orthogonal layers with exact reconstruction.

The explicit active-matrix recursion presented in Algorithm~\ref{alg:explicit}
has computational complexity $O(n^3)$, whereas the equivalent cumulative
parameter-extraction procedure requires only $O(n^2)$ operations. The squared
Frobenius norms of the layers define a nested-chain residual-energy profile that
describes how the total residual energy is distributed across the levels associated
with the fixed ordering.

This residual-energy profile provides an $\ell^2$-type structural description
of the global residual matrix. It complements both the geometric-mean priorities
obtained from the consistent component and Kii, which provides an
$\ell^\infty$-type measure determined by the worst triad defect. Entrywise
exponentiation of the additive decomposition yields an exact Hadamard-product
reconstruction of the original reciprocal matrix.

The numerical experiments confirm agreement between the two computational
procedures, exact reconstruction, mutual orthogonality, and numerical accuracy
to machine precision. They also demonstrate that the total residual energy is
invariant under simultaneous relabeling of the alternatives, that is, under
the same permutation applied simultaneously to the rows and columns, whereas
its distribution among the individual layers generally depends on the chosen
ordering. Future work will investigate principled criteria for comparing or
selecting fixed orderings.

\section*{Acknowledgements}

The authors are grateful to Dr. Julia Johnson, Professor Emeritus at Laurentian University, Sudbury, Ontario, Canada, for her careful proofreading of the manuscript and valuable suggestions on its presentation.



\end{document}